\documentclass[12pt,a4paper]{article}
\usepackage[cp1251]{inputenc}
\usepackage[english]{babel}
\usepackage[OT1]{fontenc}
\usepackage{amsmath}
\usepackage{amsfonts}
\usepackage{amssymb}
\usepackage{amsthm}
\usepackage{makeidx}
\usepackage{graphicx
}

\newtheorem{theorem}{Theorem}[section]

\theoremstyle{remark}

\def\sR{sub-Rie\-man\-ni\-an }

\newcommand{\be}[1]{\begin{equation}\label{#1}}
\newcommand{\ee}{\end{equation}}

\newcommand{\R}{\mathbb{R}}

\newcommand{\der}[2]{\frac{d \, #1}{d\, #2} }
\newcommand{\pder}[2]{\frac{\partial \, #1}{\partial \, #2} }
\newcommand{\eq}[1]{$(\protect\ref{#1})$}

\def\R{\mathbb{R}}

\def\sspan{\operatorname{span}}
\newcommand{\spann}{\operatorname{span}\nolimits}

\renewcommand{\Vec}{\operatorname{Vec}\nolimits}
\newcommand{\hall}{\operatorname{Hall}\nolimits}
\newcommand{\Lie}{\operatorname{Lie}}
\def\ds{\displaystyle}

\usepackage[all]{xy}

\def\CL{\mathcal{L}}

\title{On Carnot algebra \\with the growth vector $(2,3,5,8)$}

\author{
Yuri Sachkov\\
Program Systems Institute\\
Russian Academy of Sciences\\
Pereslavl-Zalessky,  Russia\\
E-mail: sachkov@sys.botik.ru}

\date{April 1, 2013}

\begin{document}
\maketitle

\bigskip
\centerline{\em Dedicated to Andrei Aleksandrovich Agrachev, with gratitude}
\bigskip

\begin{abstract}
We compute two vector field models of
 the  Carnot algebra with the growth vector $(2,3,5,8)$, and an infinitesimal symmetry of the corresponding \sR structure.
\end{abstract}

\tableofcontents

\section{Introduction}
Carnot groups provide a nilpotent approximation to generic \sR manifolds~\cite{gromov, mitchell, bellaiche, agrachev_sarychev, mont}. 
The free nilpotent \sR structures  are a natural first subject of study in \sR geometry starting from the growth vector $(2,3)$ --- the left-invariant \sR structure on the Heisenberg group~\cite{brock, versh_gersh}.
The next free rank 2 case --- the growth vector $(2,3,5)$  --- was studied in~\cite{dido_exp, max1, max2, max3}.

In this work we start to study the next free rank 2 case --- the growth vector $(2,3,5,8)$.
We compute two vector field models of the  corresponding Carnot algebra, and an infinitesimal symmetry of the corresponding \sR structure (see Sec. 3).

\section{Free nilpotent and Carnot Lie algebras}
\subsection{Free nilpotent Lie algebras}

Let $\CL_d$ be the real free Lie algebra with $d$ generators~\cite{reutenauer}; $\CL_d$ is the Lie algebra of commutators  of $d$ variables. We have $\CL_d=\oplus^\infty_{i=1}\CL_d^i$, where $\CL^i_d$  is the space of  commutator polynomials of degree $i$. Then 
$$\CL^{(r)}_d := \CL_d / \oplus^\infty_{i=r+1}\CL_d^i  $$
is the free nilpotent  Lie algebra of step $r$ (or of  length $r$). 

Denote
$$l_d(i):=\dim \CL^i_d, \qquad l^{(r)}_d := \dim \CL^{(r)}_d = \sum^r_{i=1}l_d(i).$$
The classical   expression of $l_d(i)$ is 
$$il_d(i)=d^i - \sum_{j|i,\ 1 \leq j<i} jl_d(j).$$

In this  work we will be interested  in free nilpotent Lie algebras with $2$ generators. Dimensions of such Lie algebras for small step are given in Table~1.   
\begin{table}[htbp]
\label{tab:dim}
\begin{center}
\begin{tabular}{|c|r|r|r|r|r|r|r|r|r|r|}
\hline
$i $&      1 & 2 & 3  & 4 & 5& 6& 7&   8& 9& 10\\
\hline
$l_2(i)$ & 2 & 1 & 2 & 3 &  6& 9& 18& 30& 56& 99\\
\hline
$l_2^{(i)}$ & 2 & 3 & 5 & 8 & 14&23& 41& 71&127&226\\
\hline
\end{tabular}
\end{center}
\caption{Dimensions of free nilpotent Lie algebras with 2 generators}
\end{table}

\subsection{Carnot algebras and groups}

A Lie algebra $L$ is called a Carnot algebra if it admits a decomposition
$$L= \oplus^r_{i=1} L_i$$
as a vector space, such that
\begin{align*}
&[L_i, L_j] \subset L_{i+ j}, \\
&L_s = {0} \text{ for } s>r,\\
&L_{i+1}=[L_1, L_i].
\end{align*}

A free nilpotent Lie algebra $\CL^{(r)}_d$ is a Carnot algebra with $L_i=\CL^i_d.$

A Carnot group $G$ is a connected, simply connected Lie group whose Lie algebra $L$ is a Carnot algebra. If $L$ is realized as the Lie algebra of left-invariant vector fields on $G$, then the degree $1$ component $L_1$ can be thought of as a completely nonholonomic (bracket-generating) distribution  on $G$. If moreover $L_1$ is endowed with a left-invariant inner product
$\left\langle \ \cdot\ , \ \cdot\ \right\rangle$, then ($G, L_1, \left\langle \ \cdot\ , \ \cdot\ \right\rangle $)  becomes a nilpotent left-invariant \sR manifold~\cite{mont}. Such \sR  structures  are nilpotent approximations of generic \sR  structures~\cite{gromov, mitchell, bellaiche, agrachev_sarychev}.

The sequence of numbers
$$ \left( \dim L_1, \dim L_1+\dim L_2, \dots, \dim L_1 + \dots + \dim L_r = \dim L \right)$$
is called the growth vector of the distribution $ L_1$~\cite{versh_gersh}.

For free nilpotent Lie algebras, the growth vector is maximal compared with all Carnot algebras with the bidimension $(\dim L_1, \dim L)$.

In this work we consider the Carnot algebra with the growth vector (2, 3, 5, 8).

\section{Lie algebra with the growth vector $(2, 3, 5, 8)$}

The Carnot algebra with the growth vector (2, 3, 5, 8) 
$$\CL^{(4)}_2 = \sspan (X_1, \dots, X_8)$$
is determined by the following  multiplication table:
\begin{align}
[X_1, X_2] &= X_3, \label{X1X2}\\
[X_1, X_3] &= X_4, \quad  [X_2, X_3] = X_5, \label{X1X3}\\
[X_1, X_4] &= X_6, \quad [X_1, X_5] = [X_2, X_4] = X_7, \quad [X_2, X_5]= X_8, \label{X1X5}
\end{align}
 with all the rest brackets equal to zero. 


\subsection{Hall basis}

Free nilpotent Lie algebras have a convenient basis introduced by M. Hall~\cite{hall}.
We describe it using the exposition of~\cite{grayson_grossman1}.

The Hall basis of the free Lie algebra $\CL_d$ with $d$ generators $X_1$, \dots, $X_d$
is the subset $\hall \subset \CL_d$ that has a decomposition into homogeneous components
$$
\hall = \cup_{i=1}^{\infty} \hall_i
$$
defined as follows.

Each element $H_j$, $j = 1, 2, \dots$, of the Hall basis  is a monomial in the generators $X_i$ and is defined recursively as follows. The generators satisfy the inclusion
$$
X_i \in \hall_1, \qquad i = 1, \dots, d,
$$
and we denote
$$
H_i = X_i, 
\qquad i = 1, \dots, d.
$$
If we have defined basis elements 
$$
H_1, \dots, H_{N_{p-1}} \in \oplus_{j=1}^{p-1} \hall_j,
$$
they are simply ordered so that $E < F$ if $E \in \hall_k$, $F \in \hall_l$, $k < l$:
$$
H_1 < H_2 < \dots < H_{N_{p-1}}.
$$
Also if $E \in \hall_s$, $F \in \hall_t$ and $p = s +t$, then
$$
[E,F] \in \hall_p
$$
if:
\begin{enumerate}
\item
$E > F$, and
\item
if $E = [G,K]$, hen $K \in \hall_q$ and $t \geq q$.
\end{enumerate}

By this definition, one easily computes recursively the first components $\hall_i$ 
of the Hall basis
for $d = 2$:
\begin{align*}
&\hall_1 = \{H_1, H_2\}, \qquad H_1 = X_1, \quad H_2 = X_2, \\
&\hall_2 = \{H_3\}, \qquad H_3 = [X_2, X_1], \\
&\hall_3 = \{H_4, H_5\}, \qquad H_4 = [[X_2, X_1],X_1], \quad  H_5 = [[X_2, X_1],X_2], \\
&\hall_4 = \{H_6, H_7, H_8\}, \\ 
&H_6 = [[[X_2, X_1],X_1],X_1], \  H_7 = [[[X_2, X_1],X_1],X_2], \  H_8 = [[[X_2, X_1],X_2],X_2].
\end{align*}
Consequently,
$$
\CL_2^{(4)} = \spann\{H_1, \dots, H_8\}.
$$
In the sequel we use a more convenient basis
$$
\CL_2^{(4)} = \spann\{X_1, \dots, X_8\}
$$
with the multiplication table~\eq{X1X2}--\eq{X1X5}.

\subsection{Asymmetric vector field model for $\CL_2^{(4)}$}
Here we recall an algorithm for construction of a vector field model for the Lie algebra $\CL_2^{(r)}$ due to Grayson and Grossman~\cite{grayson_grossman1}.
For a given $r \geq 1$, the algorithm evaluates two polynomial vector fields $H_1, H_2 \in \Vec(\R^N)$, $N = \dim \CL_2^{(r)}$, which generate the Lie algebra $\CL_2^{(r)}$. 

Consider the Hall basis elements
$$
\spann\{H_1, \dots, H_N\} = \CL_2^{(r)}.
$$
Each element $H_i \in \hall_j$ is a Lie bracket of length $j$:
\begin{align*}
&H_i = [\dots[[H_2, H_{k_j}],H_{k_{j-1}}], \dots, H_{k_1}],\\
&k_j = 1, \qquad k_{n+1} \leq k_n \text{ for } 1 \leq n \leq j-1.
\end{align*} 
This defines a partial ordering of the basis elements. We say that $H_i$ is a direct descendant of $H_2$ and of each $H_{k_l}$ and write $i \succ 2$, $i \succ k_l$, $l = 1, \dots, j$.

Define monomials $P_{2,k}$ in $x_1$, \dots, $x_N$ inductively by 
$$
P_{2,k} = - x_j \ P_{2,i} /(\deg_j P_{2,i} + 1),
$$
  whenever $H_k = [H_i,H_j]$ is a basis Hall element, and where  $\deg_j P$ is the highest power of $x_j$ which divides $P$. 

The following theorem gives the properties of the generators.

\begin{theorem}[Th. 3.1 \cite{grayson_grossman1}]\label{th:Hall}
Let $r \geq 1$ and let $N = \dim  \CL_2^{(r)}$. Then the vector fields
$$
H_1 = \pder{}{x_1}, \qquad H_2 = \pder{}{x_2} + \sum_{i \succ 2} P_{2,i} \pder{}{x_i}
$$
have the following properties:
\begin{enumerate}
\item
they are homogeneous of weight one with respect to the grading
$$
\R^N = \hall_1 \oplus \dots \oplus \hall_r;
$$
\item
$\Lie(H_1, H_2) = \CL_2^{(r)}$.
\end{enumerate}
\end{theorem}

The algorithm described before Theorem~\ref{th:Hall} produces the following vector
field basis of $\CL_2^{(4)}$:
\begin{align*}
    H_1 &= \frac{\partial }{\partial x_1},\\
    H_2 &= \frac{\partial }{\partial x_2} - x_1\frac{\partial }{\partial x_3} - \frac{x_1^2}{2}\frac{\partial }{\partial x_4}
         - x_1x_2\frac{\partial }{\partial x_5} +  \frac{x_1^3}{6}\frac{\partial }{\partial x_6} + \frac{x_1^2x_2}{2}\frac{\partial }{\partial x_7}
         + \frac{x_1x_2^2}{2}\frac{\partial }{\partial x_8},\\
    H_3 &= \frac{\partial }{\partial x_3} + x_1\frac{\partial }{\partial x_4} + x_2\frac{\partial }{\partial x_5}
         - \frac{x_1^2}{2}\frac{\partial }{\partial x_6} - x_1x_2\frac{\partial }{\partial x_7}
         - \frac{x_2^2}{2}\frac{\partial }{\partial x_8},\\
    H_4 &= -\frac{\partial }{\partial x_4} + x_1\frac{\partial }{\partial x_6} + x_2\frac{\partial }{\partial x_7},\\
    H_5 &= -\frac{\partial }{\partial x_5} + x_1\frac{\partial }{\partial x_7} + x_2\frac{\partial }{\partial x_8},\\
    H_6 &= -\frac{\partial }{\partial x_6},\\
    H_7 &= -\frac{\partial }{\partial x_7},\\
    H_8 &= -\frac{\partial }{\partial x_8},
\end{align*}
with the multiplication table
\begin{align}
    \left[ H_2, H_1 \right] &= H_3, \label{H2H1}\\
    \left[ H_3, H_1 \right] &= H_4, \; \left[ H_3, H_2 \right] = H_5, \label{H3H1} \\
    \left[ H_4, H_1 \right] &= H_6, \; \left[ H_4, H_2 \right] = H_7, \; 
    \left[ H_5, H_2 \right] = H_8. \label{H4H1}
\end{align}

\subsection{Symmetric vector field model of $\CL_2^{(4)}$}
The vector field model of the Lie algebra $\CL_2^{(4)}$ via the fields $H_1,\dotsc, H_8$
obtained in the previous subsection is asymmetric in the sense that
there is no visible symmetry between the vector fields $H_1$ and $H_2$.
Moreover, no continuous symmetries of the \sR structure
generated by the orthonormal frame $\left\{ H_1, H_2 \right\}$ are visible,
although the Lie brackets \eqref{H2H1}--\eqref{H4H1} suggest that this 
sub-Riemannian structure should be preserved by a one-parameter group of
rotations in the plane $\sspan\{H_1, H_2\}$.

One can find a symmetric vector field model of $\CL_2^{(4)}$ free of such
shortages as in the following statement.

\begin{theorem} \label{th:Xi}
    \begin{itemize}
        \item[$(1)$] 
            The vector fields
\begin{align}
&X_1 = \pder{}{x_1} - \frac{x_2}{2} \pder{}{x_3} - \frac{x_1^2 + x_2^2}{2} \pder{}{x_5} - \frac{x_1 x_2^2}{4} \pder{}{x_7} - \frac{x_2^3}{6} \pder{}{x_8}, \label{X1} \\
&X_2 = \pder{}{x_2} + \frac{x_1}{2} \pder{}{x_3} + \frac{x_1^2 + x_2^2}{2} \pder{}{x_4} + \frac{x_1^3}{6} \pder{}{x_6} + \frac{x_1^2 x_2}{4} \pder{}{x_7}, \label{X2}\\
&X_3 = \pder{}{x_3} + x_1 \pder{}{x_4} + x_2 \pder{}{x_5} + \frac{x_1^2}{2} \pder{}{x_6} + x_1 x_2 \pder{}{x_7} + \frac{x_2^2}{2} \pder{}{x_8}, \label{X3}\\
&X_4 = \pder{}{x_4} + x_1 \pder{}{x_6} + x_2 \pder{}{x_7}, \label{X4}\\
&X_5 = \pder{}{x_5} + x_1 \pder{}{x_7} + x_2 \pder{}{x_8}, \label{X5}\\
&X_6 = \pder{}{x_6}, \label{X6}\\
&X_7 = \pder{}{x_7}, \label{X7}\\
&X_8 = \pder{}{x_8} \label{X8}
\end{align}         
            satisfy the multiplication table~\eq{X1X2}--\eq{X1X5}.
            Thus the fields $X_1,\dotsc, X_8 \in \Vec(\R^8)$ model the Lie
            algebra $\CL_2^{(4)}$.
        \item[$(2)$] 
            The vector field 
\begin{align}
&X_0 = x_2 \pder{}{x_1} - x_1 \pder{}{x_2} + x_5 \pder{}{x_4} - x_4 \pder{}{x_5} + P \pder{}{x_6} + Q \pder{}{x_7} + R \pder{}{x_8},  \label{X0}\\
&P = - \frac{x_1^4}{24} +   \frac{x_1^2 x_2^2}{8} + x_7, \label{P}\\ 
&Q = \frac{x_1 x_2^3}{12} +   \frac{x_1^3 x_2}{12} - 2 x_6 + 2 x_8, \label{Q}\\
&R = \frac{x_1^2 x_2^2}{8} - \frac{x_2^4}{24} -    x_7 \label{R}
\end{align}
            satisfies the following relations:
\begin{align}
&[X_0, X_1] = X_2, \qquad   [X_0, X_2] = - X_1, \qquad [X_0, X_3] = 0, \label{X0X1}\\
&[X_0, X_4] = X_5, \qquad   [X_0, X_5] = - X_4, \label{X0X4}\\
&[X_0, X_6] = 2 X_7, \qquad   [X_0, X_7] = X_8 - X_6, \qquad   [X_0, X_8] = - 2 X_7. \label{X0X6}
\end{align}
            Thus the field $X_0$ is an infinitesimal symmetry of the sub-Riemannian structure
            generated by the orthonormal frame $\left\{ X_1, X_2 \right\}$. 
    \end{itemize}
\end{theorem}

\begin{proof}

    In fact, the both statements of the proposition are verified by the direct
computation, but we prefer to describe a method of construction of the
vector fields $X_1,\dotsc, X_8$, and $X_0$.

$(1)$
In the previous work \cite{dido_exp} we constructed a similar symmetric
vector field model for the Lie algebra $\CL_2^{(3)}$, which has growth vector (2, 3, 5):
\begin{align}
    \CL_2^{(3)} &= \sspan\{X_1,\dotsc, X_5\} \subset \Vec(\R^5), \\
     X_1 &=  \pder{}{x_1} - \frac{x_2}{2} \pder{}{x_3} - \frac{x_1^2 + x_2^2}{2} \pder{}{x_5}, \label{X1L23}\\
     X_2 &= \pder{}{x_2} + \frac{x_1}{2} \pder{}{x_3} + \frac{x_1^2 + x_2^2}{2} \pder{}{x_4},\label{X2L23}\\
     X_3 &= \pder{}{x_3} + x_1 \pder{}{x_4} + x_2 \pder{}{x_5},\label{X3L23}\\
     X_4 &= \pder{}{x_4}, \label{X4L23}\\
     X_5 &= \pder{}{x_5}, \label{X5L23}
\end{align}
with the Lie brackets~\eq{X1X2}, \eq{X1X3}. 
Now we aim to ``continue'' these relationships to vector fields
$X_1,\dotsc, X_8 \in \Vec(\R^8)$ that span the Lie algebra $\CL_2^{(4)}$.
So we seek for vector fields of the form
\begin{align}
    X_1 &= \frac{\partial }{\partial x_1} - \frac{x_2}{2}\frac{\partial }{\partial x_3} - \frac{x_1^2 + x_2^2}{2}\frac{\partial }{\partial x_5}
    + \sum\limits_{i=6}^8 a_1^i\frac{\partial }{\partial x_i},  \; \label{X1aij}   \\
    X_2 &= \frac{\partial }{\partial x_2} + \frac{x_1}{2}\frac{\partial }{\partial x_3} - \frac{x_1^2 + x_2^2}{2}\frac{\partial }{\partial x_4}
    + \sum\limits_{i=6}^8 a_2^i\frac{\partial }{\partial x_i}, \;  \label{X2aij} \\
    X_3 &= \frac{\partial }{\partial x_3} + x_1\frac{\partial }{\partial x_4} + x_2\frac{\partial }{\partial x_5}
    + \sum\limits_{i=6}^8 a_3^i\frac{\partial }{\partial x_i}, \; \label{X3aij} \\ 
    X_4 &= \frac{\partial }{\partial x_4} + \sum\limits_{i=6}^8 a_4^i\frac{\partial }{\partial x_i}, \; \label{X4aij} \\ 
    X_5 &= \frac{\partial }{\partial x_5} + \sum\limits_{i=6}^8 a_5^i\frac{\partial }{\partial x_i}, \; \label{X5aij}\\ 
    X_j &= \sum\limits_{i=i}^j a_i^j\frac{\partial }{\partial x_j}, \qquad j = 6, 7, 8, 
\end {align}
such that $\sspan\{X_1,\dotsc, X_8\} = \CL_2^{(4)}$.

Compute the required Lie brackets:

\begin{align*}
    \left[ X_1, X_2 \right] &= \frac{\partial }{\partial x_3} + x_1\frac{\partial }{\partial x_4} + x_2\frac{\partial }{\partial x_5}
    + \left( \frac{\partial a_2^6}{\partial x_1} - \frac{\partial a_1^6}{\partial x_2} \right)\frac{\partial }{\partial x_6}\\
    &+ \left( \frac{\partial a_2^7}{\partial x_1} - \frac{\partial a_1^7}{\partial x_2} \right)\frac{\partial }{\partial x_7}
    + \left( \frac{\partial a_2^8}{\partial x_1} - \frac{\partial a_1^8}{\partial x_2} \right)\frac{\partial }{\partial x_8},\\
    \left[ X_1, X_3 \right] &= \frac{\partial }{\partial x_4} + \frac{\partial a_3^6}{\partial x_1}\frac{\partial }{\partial x_6}
    + \frac{\partial a_3^7}{\partial x_1}\frac{\partial }{\partial x_7} + \frac{\partial a_3^8}{\partial x_1}\frac{\partial }{\partial x_8},\\
    \left[ X_2, X_3 \right] &= \frac{\partial }{\partial x_5} + \frac{\partial a_3^6}{\partial x_2}\frac{\partial }{\partial x_6}
    + \frac{\partial a_3^7}{\partial x_2}\frac{\partial }{\partial x_7} + \frac{\partial a_3^8}{\partial x_2}\frac{\partial }{\partial x_8},\\
    \left[ X_1, X_4 \right] &= \frac{\partial a_4^6}{\partial x_1}\frac{\partial }{\partial x_6} + \frac{\partial a_4^7}{\partial x_1}\frac{\partial }{\partial x_7}
    + \frac{\partial a_4^8}{\partial x_1}\frac{\partial }{\partial x_8},\\
    \left[ X_1, X_5 \right] &= \frac{\partial a_5^6}{\partial x_1}\frac{\partial }{\partial x_6} + \frac{\partial a_5^7}{\partial x_1}\frac{\partial }{\partial x_7}
    + \frac{\partial a_5^8}{\partial x_1}\frac{\partial }{\partial x_8},\\
    \left[ X_2, X_4 \right] &= \frac{\partial a_4^6}{\partial x_2}\frac{\partial }{\partial x_6} + \frac{\partial a_4^7}{\partial x_2}\frac{\partial }{\partial x_7}
    + \frac{\partial a_4^8}{\partial x_2}\frac{\partial }{\partial x_8},\\
    \left[ X_2, X_5 \right] &= \frac{\partial a_5^6}{\partial x_2}\frac{\partial }{\partial x_6} + \frac{\partial a_5^7}{\partial x_2}\frac{\partial }{\partial x_7}
    + \frac{\partial a_5^8}{\partial x_2}\frac{\partial }{\partial x_8}.
\end{align*}

The vector fields $X_1,\dotsc, X_8$ should be independent, thus the determinant constructed of these vectors as columns should satisfy the inequality
\begin{align*}
D = \det \left(X_1,\dotsc, X_8\right) = 
\begin{vmatrix}
a_6^6&a_7^6&a_8^6\\
a_6^7&a_7^7&a_8^7\\
a_6^8&a_7^8&a_8^8
\end{vmatrix} \neq 0.
\end{align*}
We will choose $a_i^j$ such that $D = 1$. It follows from the multiplication table
for $X_1,\dotsc, X_8$ that
\begin{align*}
D =  
\begin{vmatrix}
\ds    \frac{d^2a_3^6}{dx_1^2} & \ds\frac{d^2a_3^6}{dx_1dx_2}  & \ds\frac{d^2a_3^6}{dx_2^2}\\
\ds    \frac{d^2a_3^7}{dx_1^2} & \ds\frac{d^2a_3^7}{dx_1dx_2}  & \ds\frac{d^2a_3^7}{dx_2^2}\\
 \ds   \frac{d^2a_3^8}{dx_1^2} & \ds\frac{d^2a_3^8}{dx_1dx_2}  & \ds\frac{d^2a_3^8}{dx_2^2} 
\end{vmatrix}.
\end{align*}
In order to get $D = 1$, define the entries of this matrix as following symmetric way:
\begin{align*}
    a_3^6 = \frac{x_1^2}{2}, \qquad a_3^7 = x_1x_2, \qquad a_3^8 = \frac{x_2^2}{2}.  
\end{align*}
Then we obtain from the multiplication table for $X_1,\dotsc, X_8$ that
\begin{align*}
    \frac{\partial a_2^6}{\partial x_1} - \frac{\partial a_1^6}{\partial x_2} &= a_3^6 = \frac{x_1^2}{2},\\  
    \frac{\partial a_2^7}{\partial x_1} - \frac{\partial a_1^7}{\partial x_2} &= a_3^7 = x_1x_2,\\  
    \frac{\partial a_2^8}{\partial x_1} - \frac{\partial a_1^8}{\partial x_2} &= a_3^8 = \frac{x_2^2}{2}.  
\end{align*}
We solve these equations in the following symmetric way:
\begin{align*}
    a_1^6 &= 0, \qquad  a_2^6 = \frac{x_1^3}{6},\\ 
    a_1^7 &= -\frac{x_1x_2^2}{4}, \qquad   a_2^7 = \frac{x_1^2x_2}{4},\\
    a_1^8 &= -\frac{x_2^3}{6}, \qquad  a_2^8 = 0.  
\end{align*}
Then we substitute these coefficients to~\eq{X1aij}, \eq{X2aij} 
and check item (1) of this theorem by direct computation.

Now we prove item $(2)$.
We proceed exactly as for item $(1)$: we start from an infinitesimal symmetry~\cite{dido_exp}
\be{X0L23}
X_0 = x_2 \pder{}{x_1} - x_1 \pder{}{x_2} + x_5 \pder{}{x_4} - x_4 \pder{}{x_5} \in \Vec(\R^5)
\ee
of the \sR structure on $\R^5$ determined by the orthonormal frame \eq{X1L23}, \eq{X2L23} and ``continue'' symmetry~\eq{X0L23} to the \sR structure on $\R^8$ determined by the orthonormal frame~\eq{X1}, \eq{X2}.

So we seek for a vector field $X_0 \in \Vec(\R^8)$ of the form~\eq{X0} for the functions $P, Q, R \in C^{\infty}(\R^8)$ to be determined so that the multiplication table~\eq{X0X1}--\eq{X0X6} hold.

The first two equalities in~\eq{X0X1} yield
$$
X_1 P = - \frac{x_1^3}{6}, \qquad X_2 P = \frac{x_1^2 x_2}{2}.
$$ 
Further,
$$
X_3 P = [X_1, X_2] P = X_1 X_2 P - X_2 X_1 P = X_1 \frac{x_1^2 x_2}{2} + X_2 \frac{x_1^3}{6} = x_1 x_2.
$$
Similarly it follows that
$$
X_4 P = x_2, \quad 
X_5 P = x_1, \quad
X_6 P = 0, \quad  
X_7 P = 1, \quad 
X_8 P = 0. 
$$ 
Since $X_6 P = X_8 P = 0$, then $P = P(x_1, x_2, x_3, x_4, x_5, x_7)$. Moreover, since $X_7 P = 1$, then $P =  x_7 + a(x_1, x_2, x_3, x_4, x_5)$.
The equality
$X_5 P = x_1$ implies that $\pder{a}{x_5} = 0$, i.e., $a =   a(x_1, x_2, x_3, x_4)$.
Similarly, 
since
$X_4 P = x_2$, then  $a =   a(x_1, x_2, x_3)$.
It follows from the equality
$X_3 P = x_1x_2$ that $\pder{a}{x_3} = x_1x_2$, i.e.,   $a = x_1x_2x_3 +   b(x_1, x_2)$.
Moreover, the equality
$X_2 P = \frac{x_1^2x_2}{2}$ implies that $\pder{b}{x_2} = -x_1x_3 - \frac{x_1^2x_2}{4}$, i.e.,   $b = -x_1x_2x_3 -\frac{x_1^2x_2^2}{8} +   c(x_1)$.
Finally, the equality
$X_1 P = -\frac{x_1^3}{2}$ implies that $\der{c}{x_1} = -\frac{x_1^3}{6} + \frac{x_1x_2^2}{2}$
i.e.,
$c = -\frac{x_1^4}{24} + \frac{x_1^2x_2^2}{4}$.
Thus equality~\eq{P} follows. Similarly we get equalities~\eq{Q}, \eq{R}.

Then multiplication table~\eq{X0X1}--\eq{X0X6} for the vector field~\eq{X0}--\eq{R} is verified by a direct computation.  
\end{proof}

\section{Conclusion and future work}
We plan to perform a further study of the nilpotent \sR structure with the growth vector $(2,3,5,8)$ using its model obtained in Th.~\ref{th:Xi}:
\begin{itemize}
\item
describe the multiplication rule on the corresponding Carnot group $\R^8$,
\item
characterize Casimir functions and orbits of the co-adjoint action,
\item
describe abnormal extremal trajectories and prove strict abnormality of some of them,
\item
study symmetries, integrals and integrability of the Hamiltonian system for normal extremals,
\item
describe normal extremal trajectories.
\end{itemize}
These results will be published elsewhere.

\end{document}